\documentclass[12pt, british]{article}

\usepackage{babel}
\usepackage{amsmath,amsfonts,amssymb,amsthm}
\usepackage[utf8x]{inputenc} % UCS' UTF-8 driver is better than the LaTeX kernel's
\usepackage[T1]{fontenc} % The default font encoding only contains Latin characters
\usepackage{enumerate}
\usepackage{textcomp}
\usepackage[obeyspaces,lowtilde]{url}
\usepackage{eucal}

\usepackage[noBBpl]{mathpazo}
\usepackage{mathabx}

\usepackage[matrix,arrow]{xy}

\makeatletter

\renewcommand{\p@enumii}{}

\def\@enum@{\list{\csname label\@enumctr\endcsname}%
           {\usecounter{\@enumctr}\def\makelabel##1{
\normalfont\ignorespaces\emph{{##1}~}}
\setlength{\labelsep}{3pt}
\setlength{\parsep}{0pt}
\setlength{\itemsep}{0pt}
\setlength{\leftmargin}{0pt}
\setlength{\labelwidth}{0pt}
\setlength{\listparindent}{\parindent}
\setlength{\itemsep}{0pt}
\setlength{\itemindent}{0pt}
\topsep=3pt plus 1pt minus 1 pt}}

\makeatother

\renewcommand{\epsilon}{\ensuremath{\varepsilon}}
\renewcommand{\phi}{\ensuremath{\varphi}}

\renewcommand{\to}{\ensuremath{\longrightarrow}}

\newcommand{\N}{\ensuremath{\mathbb N}}
\newcommand{\Z}{\ensuremath{\mathbb Z}}
\newcommand{\dt}{\ensuremath{\mathbb D}^{2}}
\newcommand{\St}[1][2]{\ensuremath{\mathbb S}^{#1}}

\newcommand{\rp}{\ensuremath{\mathbb{R}P^2}}
\newcommand{\ord}[1]{\ensuremath{\left\lvert #1\right\rvert}}

\newcommand{\sn}[1][n]{\ensuremath{S_{{#1}}}}
\newcommand{\an}[1][n]{\ensuremath{A_{{#1}}}}

\DeclareRobustCommand*{\up}[1]{\textsuperscript{#1}}
\renewcommand{\th}{\ensuremath{\up{th}}}
\newcommand{\ft}[1][n]{\ensuremath{\Delta_{#1}^{2}}}
\newcommand{\garside}[1][n]{\ensuremath{\Delta_{#1}}}

\renewcommand{\ker}[1]{\ensuremath{\operatorname{\text{Ker}}\left({#1}\right)}}
\newcommand{\im}[1]{\ensuremath{\operatorname{\text{Im}}\left({#1}\right)}}

\newcommand{\id}{\ensuremath{\operatorname{\text{Id}}}}

\newcommand{\quat}[1][8]{\ensuremath{\mathcal{Q}_{#1}}}

\newcommand{\mcgst}[1][n]{\ensuremath{\operatorname{\mathcal{MCG}}(\St,#1)}}
\newcommand{\mcgrp}[1][n]{\ensuremath{\operatorname{\mathcal{MCG}}(\rp,#1)}}
\newcommand{\mcgd}[1][n]{\ensuremath{\operatorname{\mathcal{MCG}}(\dt,#1)}}
\newcommand{\mcgdb}[1][n]{\ensuremath{\operatorname{\mathcal{MCG}}_{\partial}(\dt,#1)}}
\newcommand{\mcg}[1][n]{\ensuremath{\operatorname{\mathcal{MCG}}(M,#1)}}
\newcommand{\pmcgst}[1][n]{\ensuremath{\operatorname{\mathcal{P\!MCG}}(\St,#1)}}
\newcommand{\pmcgrp}[1][n]{\ensuremath{\operatorname{\mathcal{P\!MCG}}(\rp,#1)}}

\newcommand{\pmcgdb}[1][n]{\ensuremath{\operatorname{\mathcal{P\!MCG}}_{\partial}(\dt,#1)}}
\newcommand{\pmcg}[1][n]{\ensuremath{\operatorname{\mathcal{P\!MCG}}(M,#1)}}

\makeatletter
\def\@map#1#2[#3]{\mbox{$#1 \colon\thinspace #2 \to #3$}}
\def\map#1#2{\@ifnextchar [{\@map{#1}{#2}}{\@map{#1}{#2}[#2]}}
\makeatother

\newcommand{\brak}[1]{\ensuremath{\left\{ #1 \right\}}}
\newcommand{\ang}[1]{\ensuremath{\left\langle #1\right\rangle}}
\newcommand{\normcl}[1]{\ensuremath{\ang{\!\ang{#1}\!}}}

\newcommand{\setangl}[2]{\ensuremath{\ang{\left. #1 \,\right\rvert \, #2}}}

\newcommand{\setr}[2]{\ensuremath{\brak{#1 \,\left\lvert \, #2 \right.}}}
\newcommand{\setl}[2]{\ensuremath{\brak{\left. #1 \,\right\rvert \, #2}}}

\newtheoremstyle{theoremm}{}{}{\itshape}{}{\scshape}{.}{ }{}

\theoremstyle{theoremm}
\newtheorem{thm}{Theorem}

\newtheorem{prop}[thm]{Proposition}

\newtheoremstyle{remarkk}{}{}{}{}{\scshape}{.}{ }{}

\theoremstyle{remarkk}

\newtheorem{rem}[thm]{Remark}

\newcommand{\reth}[1]{Theorem~\protect\ref{th:#1}}

\newcommand{\repr}[1]{Proposition~\protect\ref{prop:#1}}

\newcommand{\resec}[1]{Section~\protect\ref{sec:#1}}

\newcommand{\req}[1]{equation~(\protect\ref{eq:#1})}
\newcommand{\reqref}[1]{(\protect\ref{eq:#1})}

\begin{document}

\title{Minimal generating and normally generating sets for the braid and mapping class groups of $\dt$, $\St$ and $\rp$}
\author{DACIBERG~LIMA~GON\c{C}ALVES\\
Departamento de Matem\'atica - IME-USP,\\
Caixa Postal~66281~-~Ag.~Cidade de S\~ao Paulo,\\ 
CEP:~05314-970 - S\~ao Paulo - SP - Brazil.\\
e-mail:~\texttt{dlgoncal@ime.usp.br}\vspace*{4mm}\\
JOHN~GUASCHI\\
Laboratoire de Math\'ematiques Nicolas Oresme UMR CNRS~\textup{6139},\\
Universit\'e de Caen Basse-Normandie BP 5186,\\
14032 Caen Cedex, France.\\
e-mail:~\texttt{john.guaschi@unicaen.fr}}

\date{6th January 2012}

\begingroup%Localising the change to `thefootnote'.
\renewcommand{\thefootnote}{}%Removing the footnote symbol.
\endgroup 

%\keywords{braid groups, configuration space,  exact sequence, lower central series, derived series}

\maketitle

\begin{abstract}\noindent
\emph{We consider the (pure) braid groups $B_{n}(M)$ and $P_{n}(M)$, where $M$ is the $2$-sphere $\St$ or the real projective plane $\rp$. We determine the minimal cardinality of (normal) generating sets $X$ of these groups, first when there is no restriction on $X$, and secondly when $X$ consists of elements of finite order. This improves on results of Berrick and Matthey in the case of $\St$, and extends them in the case of $\rp$. We begin by recalling the situation for the Artin braid groups ($M=\dt$). As applications of our results, we answer the corresponding questions for the associated mapping class groups, and we show that for $M=\St$ or $\rp$, the induced action of $B_n(M)$ on $H_3(\widetilde{F_n(M)};\Z)$ is trivial, $F_{n}(M)$ being the $n\up{th}$ configuration space of $M$.}
%
%
%In this note we will show that the group $B_{n}(\rp)$ is generated by  two elements of finite order, namely $x=\rho_n\sigma_{n-1}....\sigma_1$ 
%and  $y=\rho_{n-1}\sigma_{n-2}....\sigma_1$ which have order $4n$, $4(n-1)$, respectively, and we show that this group is not strongly generated. Then we show similar results for  $P_{n}(\rp)$ (the pure braids on $\rp$),  for the   mapping 
%class group and the pure   mapping  class group  of  $\rp$. \comment{rewrite this later.}
% For the disc 
%$D\subset R^2$ and for the 2-sphere $\St$ we review  the known results and we complement 
%with some new results.  Also  we obtain   for  the projective plane $\rp$ 
% results to the ones for $D\subset R^2$ and for the 2-sphere $\St$. 
%As an application we will show that the fundamental group of the correspondent orbit space of the configuration spaces  $F_n(\rp)/\Sigma_n$ and $F_n(\St)/\Sigma_n$ acts on the universal covering by  orientation preserving homeomorphisms.
% For the disc 
%$D\subset R^2$ and for the 2-sphere $\St$ this has been studied. We beginthe work by reviewing   and complement the known results in this case.}
\end{abstract}

\section{Introduction}\label{sec:intro}

The braid groups $B_n$ of the plane were introduced by E.~Artin in~1925~\cite{A1,A2}. Braid groups of surfaces were studied by Zariski~\cite{Z}. They were later generalised by Fox to braid groups of arbitrary topological spaces via the following
definition~\cite{FoN}. Let $M$ be a compact, connected surface, and
let $n\in\N$. We denote the set of all ordered $n$-tuples of distinct
points of $M$, known as the \emph{$n\th$ configuration space of $M$},
by:
\begin{equation*}
F_n(M)=\setr{(p_1,\ldots,p_n)}{\text{$p_i\in M$ and $p_i\neq p_j$ if $i\neq j$}}.
\end{equation*}
Configuration spaces play an important r\^ole in several branches of mathematics and have been
extensively studied, see~\cite{CG,FH} for example. 

The symmetric group $\sn$ on $n$ letters acts freely on $F_n(M)$ by
permuting coordinates. The corresponding quotient will be denoted by
$D_n(M)$. The \emph{$n\th$ pure braid group $P_n(M)$} (respectively
the \emph{$n\th$ braid group $B_n(M)$}) is defined to be the
fundamental group of $F_n(M)$ (respectively of $D_n(M)$). We thus obtain a natural short exact sequence 
\begin{equation}\label{eq:defperm}
1 \to P_{n}(M) \to B_{n}(M) \stackrel{\pi}{\to} \sn \to 1.
\end{equation}
If $\dt\subseteq \rp$ is a topological disc, there is a group homomorphism $\map
{\iota}{B_n}[B_n(M)]$ induced by the inclusion. If $\beta\in B_n$ then we shall denote its image
$\iota(\beta)$ simply by $\beta$.

Together with the $2$-sphere $\St$, the braid groups of the
real projective plane $\rp$ are of particular interest, notably because they have non-trivial centre~\cite{vB,GG2}, and torsion elements~\cite{vB,M}.
Indeed, Fadell and Van Buskirk showed that $\St$ and $\rp$ are the only compact, connected surfaces whose braid groups have torsion. Let us recall briefly some of the properties of the braid groups of these two surfaces as well as those of the $2$-disc $\dt$~\cite{A1,A2,FvB,GvB,GG2,H,M2,vB}.

Let $n\in \N$. The Artin braid group $B_{n}$, which may be interpreted as the braid group of $\dt$, is generated by $\sigma_1,\ldots,\sigma_{n-1}$ that are subject to the following relations:
%\begin{equation*}\label{eq:presnbns}
\begin{gather}
%\begin{gathered}
\text{$\sigma_{i}\sigma_{j}=\sigma_{j}\sigma_{i}$ if $\lvert i-j\rvert\geq 2$
and $1\leq i,j\leq n-1$}\label{eq:Artin1}\\
\text{$\sigma_{i}\sigma_{i+1}\sigma_{i}=\sigma_{i+1}\sigma_{i}\sigma_{i+1}$ for
all $1\leq i\leq n-2$.}\label{eq:Artin2}
\end{gather}
The sphere braid group $B_{n}(\St)$ is also generated by $\sigma_1,\ldots,\sigma_{n-1}$, subject to the relations~\reqref{Artin1} and~\reqref{Artin2}, as well as the following `surface relation':
\begin{gather}
\sigma_1\cdots \sigma_{n-2}\sigma_{n-1}^2 \sigma_{n-2}\cdots \sigma_1=1.\label{eq:surface}
\end{gather}
%\end{gathered}
%\end{equation*}
Consequently, $B_n(\St)$ is a quotient of $B_n$. The first three sphere braid
groups are finite: $B_1(\St)$ is trivial, $B_2(\St)$ is cyclic of order~$2$,
and $B_3(\St)$ is isomorphic to the dicyclic group of order $12$ (the semi-direct product $\Z_{3}\rtimes \Z_{4}$ with non-trivial action). For $n\geq 4$, $B_n(\St)$ is infinite. The Abelianisation of $B_{n}$ (resp.\ $B_n(\St)$) is isomorphic to $\Z$ (resp.\ the cyclic group $\Z_{2(n-1)}$), where the Abelianisation homomorphism identifies all of the $\sigma_{i}$ to a single generator of $\Z$ (resp.\ of $\Z_{2(n-1)}$).

%The kernel of the associated projection 
%\begin{equation}\label{eq:abdefine}
%\left\{ \begin{aligned}
%\xi\colon\thinspace B_n(\St) &\to \Z_{2(n-1)}\\
%\sigma_i &\mapsto \overline{1} \quad\text{for all $1\leq i\leq n-1$}
%\end{aligned}\right.
%\end{equation}
%is the commutator subgroup
%$\Gamma_2\left(B_n(\St) \right)$. If $w\in B_n(\St)$ then $\xi(w)$ is the
%exponent sum (relative to the $\sigma_i$) of $w$ modulo $2(n-1)$. 
It is well known that $B_{n}$ is torsion free. Gillette and Van Buskirk showed that if $n\geq 3$ and $k\in \N$ then $B_n(\St)$ has an element of order $k$ if and only if $k$ divides one of $2n$, $2(n-1)$ or $2(n-2)$~\cite{GvB}. The torsion elements of $B_n(\St)$ were later characterised by
Murasugi:
\begin{thm}[Murasugi~\cite{M}]\label{th:murasugi}
Let $n\geq 3$. Then up to conjugacy, the torsion elements of $B_n(\St)$ are precisely the powers of
the following three elements:
\begin{enumerate}[(a)]
\item $\alpha_0= \sigma_1\cdots \sigma_{n-2}
\sigma_{n-1}$ (which is of order $2n$).
\item $\alpha_1=\sigma_1\cdots
\sigma_{n-2} \sigma_{n-1}^2$ (of order $2(n-1)$).
\item $\alpha_2=\sigma_1\cdots \sigma_{n-3} \sigma_{n-2}^2$ (of order
$2(n-2)$).
\end{enumerate}
\end{thm}
The three elements $\alpha_0$, $\alpha_1$ and $\alpha_2$ are
respectively $n\th$, $(n-1)\th$ and $(n-2)\th$ roots of $\ft$, where
$\ft$ is the so-called `full twist' braid of $B_n(\St)$, defined by
$\ft= (\sigma_1\cdots\sigma_{n-1})^n$. In other words:
\begin{equation}\label{eq:rootsft}
\alpha_0^{n}=\alpha_1^{n-1}=\alpha_2^{n-2}=\ft.
\end{equation}
So $B_n(\St)$ admits finite cyclic subgroups isomorphic to $\Z_{2n}$, $\Z_{2(n-1)}$ and $\Z_{2(n-2)}$. In~\cite[Theorem~3]{GG3}, we showed that $B_n(\St)$ is generated
by $\alpha_0$ and $\alpha_1$. If $n\geq 3$, $\ft[n]$ is the unique
element of $B_n(\St)$ of order $2$, and it generates the centre of
$B_n(\St)$. It is also the square of the \emph{Garside element} (or
`half twist') defined by:
\begin{equation}\label{eq:defgarside}
\garside = (\sigma_1 \cdots \sigma_{n-1}) (\sigma_1 \cdots \sigma_{n-2}) \cdots
(\sigma_1 \sigma_2)\sigma_1.
\end{equation}

Now let us turn to the braid groups of $\rp$. A presentation is given by:
\begin{prop}[\cite{vB}]\label{prop:presvb}
The following constitutes a presentation of the group $B_n(\rp)$:
\begin{enumerate}
\item[\underline{generators:}] $\sigma_1,\ldots,\sigma_{n-1},\rho_1,\ldots,\rho_n$.
\item[\underline{relations:}]\mbox{}
\begin{enumerate}[(i)]
\item relations~\reqref{Artin1} and~\reqref{Artin2}.
\item $\sigma_i\rho_j=\rho_j\sigma_i$ for $j \ne i, i+1$.
\item\label{it:rel3rp2} $\rho_{i+1}=\sigma_i^{-1}\rho_i\sigma_i^{-1}$ for $1\leq i \leq n-1$.
\item $\rho_{i+1}^{-1}\rho_i^{-1}\rho_{i+1}\rho_i=\sigma_i^2$  for $1\leq i \leq n-1$.
\item $\rho_1^2=\sigma_1\sigma_2\cdots\sigma_{n-2}\sigma_{n-1}^2\sigma_{n-2}\ldots\sigma_2\sigma_1$.
\end{enumerate}
\end{enumerate}
\end{prop}
For $n\geq 2$, the Abelianisation of $B_{n}(\rp)$ is isomorphic to $\Z_{2}\oplus \Z_{2}$, the $\sigma_{i}$ (resp.\ $\rho_{j}$) being identified to a generator of the first (resp.\ second factor). A presentation of $P_{n}(\rp)$ was given in~\cite[Theorem~4]{GG8} (note however that the generators $\rho_{i}$ given there are slightly different from those of Van Buskirk). The first two braid groups of $\rp$ are finite: $B_1(\rp)=P_1(\rp)\cong \Z_{2}$, $P_{2}(\rp)$ is isomorphic to the quaternion group $\quat$ of order~$8$, and $B_{2}(\rp)$ is isomorphic to the generalised quaternion group $\quat[16]$ of order $16$. For $n\geq 3$, $B_{n}(\rp)$ is infinite. 
%The pure braid group $P_{3}(\rp)$ is isomorphic to a semi-direct product of a free group of
%rank $2$  by $\quat$~\cite{vB}; an explicit action was given in~\cite{GG2,GG11}.
The finite order elements of $B_{n}(\rp)$ were also characterised by Murasugi~\cite[Theorem~B]{M}, however their orders were not clear, even for elements of $P_{n}(\rp)$. In~\cite[Corollary~19 and Theorem~4]{GG2}, we proved that for $n\geq 2$, the torsion of $P_{n}(\rp)$ is $2$ and $4$, and that of $B_{n}(\rp)$ is equal to the divisors of $4n$ and
$4(n-1)$, and in \cite[Section~3]{GG10} we simplified somewhat Murasugi's characterisation. 

Following~\cite[Sections~3 and~4]{GG2}, set 
\begin{equation}\label{eq:defab}
\left\{
\begin{aligned}
a &=\rho_{n}\sigma_{n-1}\cdots \sigma_{1}\\
b &=\rho_{n-1}\sigma_{n-2}\cdots \sigma_{1}
\end{aligned}
\right.
\end{equation}
and
\begin{equation}\label{eq:defalphabeta}
\left\{
\begin{aligned}
A &=a^n=\rho_{n}\cdots \rho_{1}\\
B &=b^{n-1}=\rho_{n-1}\cdots \rho_{1}.
\end{aligned}
\right.
\end{equation}
These elements play an important rôle in the analysis of $B_{n}(\rp)$, and in particular in the study of many of its finite subgroups. By~\cite[Proposition~26]{GG2}, $a$ and $b$ are of order $4n$ and $4(n-1)$ respectively, so $A$ and $B$ are of order $4$. Further, by~\cite[Proposition 10]{GG10}, any element of order $4$ in $P_{n}(\rp)$ is conjugate via an element of $B_{n}(\rp)$ to $A$ or $B$. However, the number of conjugacy classes in $P_{n}(\rp)$ of order $4$ elements was not known. A naïve upper bound for this number, $2n!$, may be obtained by multiplying the number of conjugacy classes in $P_{n}(\rp)$ by $[B_{n}(\rp):P_{n}(\rp)]$. In \repr{ccpnrp2}, we shall compute the exact value. 

If $\Gamma$ is a group, let $\Gamma\up{Ab}$ denote its Abelianisation. For $X$ a subset of $\Gamma$, let $\ang{X}$ denote the subgroup of $\Gamma$ generated by $X$, and let $\normcl{X}$ denote the normal closure of $X$ in $\Gamma$. Then $\Gamma$ is \emph{generated} (resp.\ \emph{normally generated}) by $X$ if $\Gamma=\ang{X}$ (resp.\ $\Gamma=\normcl{X}$). It is a natural question as to whether $\Gamma=\ang{X}$ is (normally) generated by a finite subset or not. If it is, one can ask the following questions: 
\begin{enumerate}
\item[\normalfont\textbf{Question~1:}]\label{it:qu1} compute
\begin{align*}
\operatorname{G}(\Gamma)&=\min \setr{\ord{X}}{\text{$\Gamma$ is generated by $X$}}, \quad \text{and} \\
\operatorname{NG}(\Gamma)&=\min \setr{\ord{X}}{\text{$\Gamma$ is normally generated by $X$}},
\end{align*}
the minimal number of elements among all (normal) generating sets of $\Gamma$. 

\item[\normalfont\textbf{Question~2:}] we can refine Question~\ref{it:qu1} if we impose additional constraints on the elements of $X$. We shall say that a group $\Gamma$ is \emph{torsion generated} (resp.\ \emph{normally torsion generated}) if there exists a subset $X$ of elements of $\Gamma$ of finite order such that $\Gamma=\ang{X}$ (resp.\ $\Gamma=\normcl{X}$). If there exists a finite set $X$ satisfying this property then one may ask to compute:
\begin{align*}
\operatorname{TG}(\Gamma)&=\min \setr{\ord{X}}{\text{$\Gamma$ is torsion generated by $X$}}, \quad\text{and}\\
\operatorname{NTG}(\Gamma)&=\min \setr{\ord{X}}{\text{$\Gamma$ is normally torsion generated by $X$}},
\end{align*}
the minimal number of elements among all (normal) generating sets of $G$ consisting of finite order elements.
\end{enumerate}
We have the following implications between the various notions:
\begin{align*}
& \text{torsion generated $\Longrightarrow$ generated $\Longrightarrow$ normally generated}\\
& \text{torsion generated $\Longrightarrow$ normally torsion generated $\Longrightarrow$ normally generated.}
\end{align*}
These relations imply that if the given numbers are defined for a group $\Gamma$ then:
\begin{align}
\operatorname{NG}(\Gamma) \leq \operatorname{G}(\Gamma)\leq \operatorname{TG}(\Gamma)\quad\text{and}\label{eq:inequgeni}\\
\operatorname{NG}(\Gamma) \leq \operatorname{NTG}(\Gamma) \leq \operatorname{TG}(\Gamma)\label{eq:inequgen}.
\end{align}
As a special case, recall from~\cite{BeMa,BeMi} that if $k\geq 2$, $\Gamma$ is said to be \emph{strongly $k$-torsion generated} if there exists an element $g_{k}\in \Gamma$ of order $k$ such that $\Gamma=\normcl{g_{k}}$. In other words, $\Gamma$ is strongly $k$-torsion generated for some $k\in \N$ if and only if $\operatorname{NTG}(\Gamma)=1$. In~\cite{BeMa}, Berrick and Matthey considered the problem of strong torsion generation for various groups, among them the braid groups of $\St$ and $\rp$, and they proved the following result.
\begin{prop}[{\cite[Proposition~5.3]{BeMa}}]\label{prop:bema}
The normal closure of the element $\alpha_{1}$ has index $\gcd(n,2)$ in $B_{n}(\St)$. In particular, for $n$ odd, $B_{n}(\St)$ is strongly $2(n-1)$-torsion generated.
\end{prop}
The question of the strong $k$-torsion generation of $B_{n}(\rp)$ is mentioned in~\cite[page~923]{BeMa}, although no result along the lines of \repr{bema} is given. The aim of this paper is to determine in general the numbers $\operatorname{G}(\Gamma)$ and $\operatorname{NG}(\Gamma)$ where $\Gamma= B_{n}(M)$ for $M=\dt,\St$ or $\rp$, and the numbers $\operatorname{TG}(\Gamma)$ and $\operatorname{NTG}(\Gamma)$ where $\Gamma$ is one of $B_{n}(\St)$ or $B_{n}(\rp)$, as well as to find generating sets that realise these quantities. In a similar spirit, we shall also discuss these problems for the corresponding pure braid groups and their Abelianisation. For $B_{n}$, the results are well known, but since we were not able to find a proof in the literature, we shall discuss this case briefly in \resec{artin}.
\begin{prop}Let $n\geq 2$.\label{prop:mingenartin}
\begin{enumerate}
\item\label{it:gensbn} $\operatorname{G}(B_n)=
\begin{cases}
1 & \text{if $n=2$}\\
2 & \text{if $n\geq 3$.}
\end{cases}$ Furthermore, for all $n\geq 2$, $\normcl{\sigma_1}=B_n$, so $\operatorname{NG}(B_n)=1$.
\item\label{it:gensbnb} $\operatorname{NG}(P_n)= \operatorname{G}(P_n)= \operatorname{G}(P_n\up{Ab})=n(n-1)/2$.
\end{enumerate}
\end{prop}

We turn to the case of the sphere in \resec{murasugi}. The following result extends that of \repr{bema}, and also treats the case of $P_{n}(\St)$. Note that if $n\geq 3$, $\ft$ is the unique torsion element of $P_{n}(\St)$, and so if $n\geq 4$, $P_{n}(\St)$ cannot be (normally) torsion generated. 
\begin{thm}\label{th:bemaplus}
Let $n\geq 3$, and for $i=0,1,2$, let $\alpha_{i}$ be as defined in \reth{murasugi}.
\begin{enumerate}
\item\label{it:bemaplusa} $\operatorname{G}(B_n(\St))=2$, $\operatorname{NG}(B_n(\St))=1$ and $\operatorname{TG}(B_n(\St))=2$.
\item\label{it:bemaplusb} If $n$ is even, there is no torsion element of $B_n(\St)$ whose normal closure is $B_n(\St)$. Furthermore, $B_n(\St)/\normcl{\alpha_1}\cong \Z_2$.
\item\label{it:bemaplusc} $\operatorname{NTG}(B_n(\St))=
\begin{cases}
1 & \text{if $n$ is odd}\\
2 & \text{if $n$ is even.}
\end{cases}$
\item\label{it:bemaplusd} The quotient of $B_n(\St)$ by the normal closure of either $\alpha_0$ or $\alpha_2$ is isomorphic to $\Z_{n-1}$, unless $n=3$ and $i=2$, in which case $B_{3}(\St)/\normcl{\alpha_{2}}\cong \sn[3]$.
\item\label{it:bemapluse} $\operatorname{G}(P_n(\St))=\operatorname{NG}(P_n(\St))=\operatorname{G}(P_n(\St)\up{Ab})=(n(n-3)+2)/2$. 
\end{enumerate}
\end{thm}
Part~(\ref{it:bemaplusc}) includes Berrick and Matthey's result given in \repr{bema}.

In \resec{rp2}, we consider the braid groups of $\rp$, and obtain the following results: 
\begin{thm}\label{th:genrp}
Let $n\geq 2$, and let $a$ and $b$ be as given in \req{defab}.
\begin{enumerate}
\item\label{it:genrpa} The group $B_{n}(\rp)$ is generated by $\brak{a,b}$, and $\operatorname{G}(B_n(\rp))=\operatorname{TG}(B_n(\rp))=2$.
\item\label{it:normgenrp2} The normal closure of any element of $B_n(\rp)$ is a proper subgroup of $B_n(\rp)$, and $\operatorname{NG}(B_n(\rp))=\operatorname{NTG}(B_n(\rp))=2$. In particular, $B_n(\rp)$ is not strongly $k$-torsion generated for any $k\in\N$.
\item\label{it:torsnormgenrp2} The quotient of $B_n(\rp)$ by either $\normcl{a}$ or by $\normcl{b}$ is isomorphic to $\Z_2$. 
\item\label{it:genrpd} The group $P_{n}(\rp)$ is torsion generated by the following set of torsion elements:
\begin{equation}\label{eq:torgenrp2}
Y=\brak{a^{n}, b^{n-1}, a^{-1}b^{n-1}a,\ldots
,a^{-(n-2)}b^{n-1}a^{n-2}}
\end{equation}
of order $4$. Further,
\begin{equation}\label{eq:genspnrp2}
\operatorname{G}(P_{n}(\rp))= \operatorname{NG}(P_{n}(\rp))=\operatorname{TG}(P_{n}(\rp))= \operatorname{NTG}(P_{n}(\rp))=n.
\end{equation}
In particular, $P_{n}(\rp)$ cannot be normally generated by any subset consisting of less than $n$ elements.
\end{enumerate}
\end{thm}
Part~(\ref{it:normgenrp2}) thus answers negatively the underlying question of~\cite{BeMa} concerning the strong $k$-torsion generation of $B_{n}(\rp)$.

In \resec{mcg}, \repr{mingenartin} and Theorems~\ref{th:bemaplus} and~\ref{th:genrp} will be applied to obtain similar results for the mapping class and pure mapping class groups of the given surfaces. As another application, in \resec{actionh3} we prove the triviality of the action on homology of the universal covering of the configuration space of the braid groups of $\St$ and $\rp$:
\begin{prop}\label{prop:acthom}
Let $M$ be equal either to $\St$ or to $\rp$, let $n\geq 3$ if $M=\St$ and $n\geq 2$ if $M=\rp$, and let $H$ be any subgroup of $B_{n}(M)$. Then the induced action of $H$ on $H_3(\widetilde{F_n(M)};\Z)\cong \Z$ is trivial. In particular the action of $P_n(M)$ and $B_n(M)$ on $H_3(\widetilde{F_n(M)};\Z)$ is trivial.  
 \end{prop}

% This paper is divided into 6 section besides the introduction. In section 2 we discuss the case where the surface is the disc $D$.  In section 3 we discuss the case where the surface is the sphere. In section 4 
%we discuss the case where the group is the braid group of the  projective plane.  In section 5 
%we discuss the case where the group is the pure braid group of the  projective plane.  In section 6 we apply the results from the previous sections  to obtain  results  for the Mapping class and the pure Mapping class of the various surfaces. In section 7 we show the triviality of the action on homology of the
% on the universal covering of the configuration space of  braid group.

\subsection*{Acknowledgements}

This work took place during the visits of the first author to the Laboratoire de Math\'ematiques Nicolas Oresme, Universit\'e de Caen Basse-Normandie during the period 5\up{th} November~--~5\up{th} December 2010, and of the second author to the Departmento de Matem\'atica do IME-Universidade de S\~ao Paulo during the periods 12\up{th}~--~25\up{th}~April and 4\up{th}~--~26\up{th} October 2010, and 24\up{th} February~--~6\up{th} March 2011. It was supported by the international Cooperation CNRS/FAPESP project number 09/54745-1, by the ANR project TheoGar project number ANR-08-BLAN-0269-02, and by the FAPESP ``Projecto Tem\'atico Topologia Alg\'ebrica, Geom\'etrica e Differencial'' number  2008/57607-6.

\section{Generating sets of $B_{n}$ and $P_{n}$}\label{sec:artin}

It is a classical fact that $B_{n}$ and $P_{n}$ are finitely presented, a presentation of the latter being given in~\cite[Appendix~1]{H}. Although the results of \repr{mingenartin} concerning the Artin braid groups are well known, we were not able to find an explicit statement in the literature, so we give a brief account here. Before doing so, we state the following result which allows us to compare the minimal number of generators of a group and its Abelianisation, and which will prove to be useful in what follows.

\begin{prop}\label{prop:mingen} 
Let $\Gamma_1,\Gamma_2$ be groups, and let $\map{\phi}{\Gamma_1}[\Gamma_2]$ be a surjective group homomorphism. If $\Gamma_1$ is finitely generated then so is $\Gamma_{2}$, and we have the inequalities $\operatorname{G}(\Gamma_1) \geq \operatorname{G}(\Gamma_2)$ and $\operatorname{NG}(\Gamma_1) \geq \operatorname{NG}(\Gamma_2)$. In particular, if $\Gamma_{2}=\Gamma_{1}\up{Ab}$ and $\phi$ is Abelianisation, $\operatorname{G}(\Gamma_1) \geq \operatorname{NG}(\Gamma_1)\geq \operatorname{G}(\Gamma_1\up{Ab})$.
\end{prop}

\begin{proof}
Straightforward.
\end{proof}

We now prove \repr{mingenartin}.

\begin {proof}[Proof of \repr{mingenartin}.]\mbox{}
\begin{enumerate}[(a)]
\item If $n=2$ then $B_{2}\cong \Z$, and the statement follows easily. The case $n\geq 3$ is a consequence of~\cite[Chapter~2, Section~2, Exercise~2.4]{M2} of which a generalisation may be found in~\cite[Lemma~29]{GG13}. Indeed, by induction we have:
\begin{equation*}
(\sigma_1\sigma_2\cdots\sigma_{n-1})^{i} \sigma_{1}(\sigma_1\sigma_2\cdots\sigma_{n-1})^{-i}=\sigma_{i+1} \quad\text{for $i=1,\ldots, n-2$,} 
\end{equation*}
which implies that $B_{n}=\ang{\sigma_1, \sigma_1\sigma_2\cdots\sigma_{n-1}}$. But $B_{n}$ is non cyclic, and so $\operatorname{G}(B_{n})=2$. This calculation also shows that $B_{n}=\normcl{\sigma_{1}}$, and thus $\operatorname{NG}(B_{n})=1$.
\item Using the presentation of $P_{n}$ given in~\cite[Appendix~1]{H}, we see that that $P_n\up{Ab}$ is a free Abelian group of rank $n(n-1)/2$   with a basis comprised of the $\Gamma_{2}(P_{n})$-cosets of the standard generators  
\begin{equation}\label{eq:defaij}
A_{i,j}=\sigma_{j-1}\cdots \sigma_{i+1}\sigma_{i}^{2} \sigma_{i+1}^{-1}\cdots \sigma_{j-1},\, 1\leq i<j\leq n,
\end{equation}
of $P_{n}$, and denoted in~\cite{H} by $a_{i,j}$ (here $\Gamma_{2}(P_{n})$ denotes the commutator subgroup of $P_{n}$). Thus $\operatorname{G}(P_n)\geq \operatorname{NG}(P_n)\geq \operatorname{G}(P_n\up{Ab})=n(n-1)/2$ by \repr{mingen}. But $P_n$ is generated by the $A_{i,j}$, so $n(n-1)/2\geq \operatorname{G}(P_n)$, and the result follows.\qedhere
\end{enumerate}
\end{proof}

\section{Generating sets of $B_{n}(\St)$ and $P_{n}(\St)$}\label{sec:murasugi}

%In this section we begin by  recall Murasugi's characterisation of the finite order elements of
%$B_{n}(\St)$. Then we select a finite number of torsion elements such that these elements generate the 
%entire group. 
%
%
%Up to conjugacy, the finite order elements of $B_{n}(\St)$ may be characterised as
%follows:
%\begin{thm}[\cite{M}]\label{th:murasugi}
%Any finite order element of $B_{n}(\St)$  is conjugate to a power of one of the following three elements:
%\begin{enumerate}[(a)]
%\item $\alpha_0=\sigma_1\cdots \sigma_{n-1}$.
%\item $\alpha_1=\sigma_1\cdots \sigma_{n-2} \sigma_{n-1}^2$.
%\item $\alpha_2=\sigma_1\cdots \sigma_{n-3} \sigma_{n-2}^2$.
%\end{enumerate}
%\end{thm}

Since the group $B_n(\St)$ is a quotient of $B_n$ and $B_{n}(\St)$ is not cyclic if $n\geq 3$, the statement of \repr{mingenartin}(\ref{it:gensbn}) remains true if we replace $B_n$ by $B_n(\St)$. But for the case of $B_n(\St)$, we can refine our analysis by looking at generating sets consisting of elements of finite order. It was proved in~\cite[Theorem~3]{GG3} that $B_n(\St)=\ang{\alpha_0,\alpha_1}$, $\alpha_0$ and $\alpha_1$ being the elements of \reth{murasugi}. By \repr{bema}, if $n$ is odd, $B_{n}(\St)=\normcl{\alpha_{1}}$, and thus $B_n(\St)$ is strongly $2(n-1)$-torsion generated in this case. If $n$ is even, the same proposition shows that $\normcl{\alpha_{1}}$ is a subgroup of $B_{n}(\St)$ of index $2$. In \reth{bemaplus}, we improve these results somewhat. We now proceed with the proof of that theorem.

\begin{proof}[Proof of \reth{bemaplus}.] \mbox{}
\begin{enumerate}[(a)]
\item Propositions~\ref{prop:mingenartin} and~\ref{prop:mingen} imply that $\operatorname{G}(B_{n}(\St))=2$ and $\operatorname{NG}(B_{n}(\St))=1$, using the fact that $B_{n}(\St)$ is non cyclic and is a quotient of $B_{n}$. Since $B_{n}(\St)=\ang{\alpha_{0}, \alpha_{1}}$~\cite[Theorem~3]{GG3}, it then follows that $\operatorname{TG}(B_{n}(\St))=2$.
\item Equations~\reqref{Artin1},~\reqref{Artin2} and~\reqref{surface} imply that for $n\geq 2$, $(B_{n}(\St))\up{Ab}\cong \Z_{2(n-1)}$, where the Abelianisation homomorphism $\map{\phi}{B_{n}(\St)}[(B_{n}(\St))\up{Ab}]$ identifies the standard generators $\sigma_{1},\ldots, \sigma_{n-1}$ of $B_{n}(\St)$ to the generator $\overline{1}$ of $\Z_{2(n-1)}$, and sends $\alpha_0, \alpha_1$ and $\alpha_2$ to  $\overline{n-1}$, $\overline{n}$ and $\overline{n-1}$ respectively.  If $i\in\brak{0,2}$, or if $i=1$ and $n$ is even then $\ang{\phi(\alpha_{i})}\neq (B_{n}(\St))\up{Ab}$, and so $\normcl{\alpha_{i}}\neq B_{n}(\St)$. The first part of the statement is then a consequence of \reth{murasugi}, and the second part follows from \repr{bema}.

\item If $n$ is odd, the result follows from \repr{bema}, while if $n$ is even, we see by \req{inequgeni} that $\operatorname{NTG}(B_{n}(\St))=2$ because $\operatorname{G}(B_{n}(\St))=2$ and $\operatorname{NTG}(B_{n}(\St))\neq 1$ by part~(\ref{it:bemaplusb}).

\item We first treat the exceptional case $n=3$ and $i=2$. In this case, $\alpha_{2}=\sigma_{1}^{2}$ is a non-trivial torsion element of $P_{3}(\St)$, so is equal to $\ft[3]$, which generates $P_{3}(\St)$. The fact that $B_{3}(\St)/\normcl{\alpha_{2}}\cong \sn[3]$ is a consequence of \req{defperm}. So suppose that $i\in\brak{0,2}$, but with $i=0$ if $n=3$. The map which sends each $\sigma_j$,  $j=1, \ldots, n-1$, to the generator $\overline{1}$ of $\Z_{n-1}$ extends to a surjective homomorphism $\map{\psi}{B_n(\St)}[\Z_{n-1}]$. Since $\psi(\alpha_{i})=\overline{0}$ and the target is Abelian, we see that $\normcl{\alpha_{i}}\subset \ker{\psi}$ and thus $\psi$ factors through $B_n(\St)/\normcl{\alpha_{i}}$. In particular, the order of $B_n(\St)/\normcl{\alpha_{i}}$ is at least $n-1$. If $i=0$, the fact that $B_{n}(\St)=\ang{\alpha_{0},\alpha_{1}}$ implies that $B_n(\St)/\normcl{\alpha_{0}}$ is generated by the coset of $\alpha_{1}$, so is cyclic. Using \req{rootsft}, it follows that $B_n(\St)/\normcl{\alpha_{0}}$ is of order at most $n-1$, from which we deduce that it is cyclic of order exactly $n-1$. Now suppose that $i=2$, so $n\geq 4$, and let $Q=B_n(\St)/\normcl{\alpha_2}$. By abuse of notation, we shall denote the projection of an element $x\in B_n(\St)$ in $Q$ by $x$. We will show that
\begin{equation}\label{eq:sigmaQ}
\sigma_1=\cdots= \sigma_{n-1} \;\text{in $Q$.} 
\end{equation}
To do so, we shall prove by induction that $\sigma_{n-2-i}=\sigma_{n-2}$ in $Q$ for all $1\leq i\leq n-3$. 
%Note that we will not make use of the surface relation~\reqref{surface}, so this equality also holds in the quotient group $B_n\bigl/\normcl{\sigma_{1}\cdots \sigma_{n-3}\sigma_{n-2}^{2}}\bigr.$. To prove the equality of the $\sigma_{i}$ in $Q$, 
For the case $i=1$, we first multiply the relation $\alpha_{2}=1$ in $Q$ on the left-hand side by $\sigma_{n-2}$, so:
\begin{align*}
\sigma_{n-2}&= \sigma_{n-2}\alpha_{2}=\sigma_{n-2} \sigma_1 \sigma_2\cdots\sigma_{n-3}\sigma_{n-2}\sigma_{n-2}= \sigma_1 \sigma_2\cdots\sigma_{n-4}\sigma_{n-2}\sigma_{n-3}\sigma_{n-2} \sigma_{n-2}\\
&=\sigma_1 \sigma_2\cdots\sigma_{n-3}\sigma_{n-2} \sigma_{n-3} \sigma_{n-2}=\alpha_{2}\sigma_{n-2}^ {-1} \sigma_{n-3} \sigma_{n-2} \quad\text{using \req{Artin2},}
\end{align*}
which in $Q$ implies that $\sigma_{n-2}=\sigma_{n-2}^ {-1} \sigma_{n-3} \sigma_{n-2}$, hence $\sigma_{n-3}=\sigma_{n-2}$. Now suppose that $\sigma_{n-2-i}=\sigma_{n-2}$  in $Q$ for some $1\leq i\leq n-4$. Since $(n-2)-(n-3-i)=1+i\geq 2$, by equations~\reqref{Artin1} and~\reqref{Artin2} and the induction hypothesis, we have:
\begin{align*}
\sigma_{n-2}&= \sigma_{n-3-i} \sigma_{n-2}\sigma_{n-3-i}^{-1}=\sigma_{n-3-i} \sigma_{n-2-i}\sigma_{n-3-i}^{-1}= \sigma_{n-2-i}^{-1}\sigma_{n-3-i} \sigma_{n-2-i}\\
&= \sigma_{n-2}^{-1}\sigma_{n-3-i} \sigma_{n-2}=\sigma_{n-3-i},
\end{align*}
from which it follows that $\sigma_{1}=\cdots=\sigma_{n-2}$ in $Q$ by induction. Since $n\geq 4$, in $Q$ we obtain similarly:
\begin{equation*}
\sigma_{n-1}= \sigma_{n-3}\sigma_{n-1}\sigma_{n-3}^{-1}=\sigma_{n-2}\sigma_{n-1}\sigma_{n-2}^{-1} =\sigma_{n-1}^{-1}\sigma_{n-2}\sigma_{n-1}= \sigma_{n-1}^{-1}\sigma_{n-3}\sigma_{n-1}=\sigma_{n-3},
\end{equation*}
which yields \req{sigmaQ}. Denoting the common element $\sigma_{i}$ in $Q$ by $\sigma$, we see that $Q$ is cyclic, generated by $\sigma$, and that $\sigma^{n-1}=1$ in $Q$ because $\alpha_{2}$ is trivial in $Q$. This implies that $Q$ is a quotient of $\Z_{n-1}$, and since $Q$ surjects onto $\Z_{n-1}$, it follows from above that $Q\cong \Z_{n-1}$ as required.

%For the second part  we look first the case of $\alpha_1$.  From \cite{GG3} Theorem 3
%we observe that   the quotient 
%of $B_n(\St)$ by the normal closure of $\alpha_1$ is generated by the image of
%$\alpha_0$, so abelian. Therefore this projection factors through the Abelianization . Since
%$\alpha_0$ on the abelianized has order 2, it follows that the quotient in question is either trivial
%of $Z_2$. From the first part the result follows.
\item By~\cite[Theorem~4(i)]{GG1}, we have the following isomorphism:
\begin{equation}\label{eq:dirsum}
P_n(\St)\cong P_{n-3}(\St \setminus \brak{x_1,x_2,x_3})\oplus \Z_2,
\end{equation}
where the $\Z_{2}$-factor is generated by $\ft$. Using the presentation given in~\cite[Proposition~7]{GG3}, $P_{n-3}(\St \setminus \brak{x_1,x_2,x_3})$ admits a generating set $X=\setl{A_{i,j}}{4\leq j\leq n,\, 1\leq i<j}$ consisting of $(n-3)(n+2)/2$ elements, where $A_{i,j}$ is as defined in \req{defaij}. For each $j=4,\ldots,n$, the surface relation $\left(\prod_{i=1}^{j-1} A_{i,j}\right)\left(\prod_{k=j+1}^{n} A_{j,k}\right)=1$ allows us to delete the generator $A_{1,j}$ from $X$, yielding a generating set $X'=\setl{A_{i,j}}{4\leq j\leq n,\, 2\leq i<j}$ consisting of $n(n-3)/2$ elements, so
\begin{equation}\label{eq:pnst}
\operatorname{G}(P_{n}(\St))\leq (n(n-3)+2)/2
\end{equation}
by \req{dirsum}. Furthermore, the set $X'$ is just now subject to the remaining Artin braid relations (the those given at the top of \cite[page~385]{GG3}) of the presentation of $P_{n-3}(\St \setminus \brak{x_1,x_2,x_3})$, rewritten in terms of the elements of $X'$. These relations may be written as commutators of elements of $X'$, and so collapse under Abelianisation. Thus $(P_{n-3}(\St \setminus \brak{x_1,x_2,x_3}))\up{Ab}\cong \Z^{n(n-3)/2}$, for which a basis is given by the cosets of the elements of $X'$. We deduce from \req{dirsum} that $(P_n(\St))\up{Ab}\cong \Z^{n(n-3)/2}\oplus \Z_2$, and so $\operatorname{G}((P_{n}(\St))\up{Ab})=(n(n-3)+2)/2$. The result then follows from \repr{mingen} and \req{pnst}.\qedhere
\end{enumerate}
\end{proof}
 
\section{Generating sets of $B_{n}(\rp)$ and $P_{n}(\rp)$}\label{sec:rp2}

We now consider the case of the braid groups of the projective plane, and prove \reth{genrp}.

\begin{proof}[Proof of \reth{genrp}.]\mbox{}
\begin{enumerate}
\item As we mentioned just after \req{defab}, $a$ and $b$ are torsion elements. Since $B_{n}(\rp)$ is non cyclic for $n\geq 2$ (its Abelianisation is $\Z_{2}\oplus \Z_{2}$ by \repr{presvb}), we thus have $\operatorname{G}(B_{n}(\rp))\geq 2$. Let $L$ be the subgroup of $B_{n}(\rp)$ generated by $\brak{a,b}$. By \req{defab}, we have:
\begin{equation*}
ab^{-1}=(\rho_n\sigma_{n-1}\cdots\sigma_1) (\rho_{n-1}\sigma_{n-2}\cdots\sigma_1)^{-1}=\rho_n\sigma_{n-1}\rho_{n-1}^{-1}=
\sigma_{n-1}^{-1},
\end{equation*}
using relation~(\ref{it:rel3rp2}) of \repr{presvb}, so $\sigma_{n-1}\in L$. From~\cite[page~777]{GG2}, $a^{j}\sigma_{n-1}a^{-j}=\sigma_{n-j-1}$ for all $j=1,\ldots,n-2$, and so $\sigma_{i}\in L$ for all $i=1,\ldots,n-1$. Hence $\rho_{n}\in L$ by \req{defab}, and again by~\cite[page~777]{GG2}, it follows that 
\begin{equation}\label{eq:conja}
a^{j}\rho_{n}a^{-j}=\rho_{n-j}\quad\text{for all $j=1,\ldots,n-1$,}
\end{equation}
and so $\rho_{i}\in L$ for all $i=1,\ldots,n$. \repr{presvb} implies that $L=B_{n}(\rp)$, and so $\operatorname{TG}(B_{n}(\rp))=2$. Equation~\reqref{inequgeni} then forces $\operatorname{TG}(B_{n}(\rp))= \operatorname{G}(B_{n}(\rp))=2$.

\item Since $(B_{n}(\rp))\up{Ab}\cong \Z_2\oplus \Z_2$, we have $\operatorname{G}((B_{n}(\rp))\up{Ab})=2$. It follows from part~(\ref{it:genrpa}), \repr{mingen}  and \req{inequgen} that $\operatorname{NG}(B_{n}(\rp))=\operatorname{NG}(B_{n}(\rp))=2$, in particular $\normcl{x}\subsetneqq B_{n}(\rp)$ for all $x\in B_{n}(\rp)$.

\item Let $x\in \brak{a,b}$. By part~(\ref{it:normgenrp2}), $\normcl{x}\neq B_{n}(\rp)$. Since $B_{n}(\rp)=\ang{a,b}$ by part~(\ref{it:genrpa}), $B_{n}(\rp)/\normcl{x}$ is cyclic, so Abelian, and thus the projection $B_{n}(\rp)\to B_{n}(\rp)/\normcl{x}$ factors through $(B_{n}(\rp))\up{Ab}\cong \Z_2\oplus \Z_2$, from which it follows that $B_{n}(\rp)/\normcl{x}\cong \Z_{2}$.

\item Recall from~\cite[Corollary~19]{GG2} that the torsion elements of $P_n(\rp)$ are of order $2$ or $4$, and the only element of order $2$ is the full twist. By arguments similar to those of~\cite[Theorem~4]{GG8}, one may see that $P_{n}(\rp)$ is generated by 
\begin{equation}\label{eq:prespnrp}
\setangl{A_{i,j},\, \rho_{k}}{1\leq i<j\leq n,\, 1\leq k\leq n},
\end{equation}
where $A_{i,j}$ is given by \req{defaij} and $\rho_{k}$ is that of \repr{presvb} (beware however that the elements $\rho_{i}$ of \cite{GG8} differ somewhat from those of \repr{presvb}), and that $(P_{n}(\rp))\up{Ab}\cong \Z_{2}^n$, where under the Abelianisation homomorphism $\map{\xi}{P_{n}(\rp)}[\Z_{2}^n]$, $\rho_{k}$ is sent to the element $(\overline{0}, \ldots, \overline{0}, \underbrace{\overline{1}}_{\text{$k\up{th}$ position}}, \overline{0}, \ldots, \overline{0})$ of $\Z_{2}^n$, and for all $1\leq i<j\leq n$, $A_{i,j}$ is sent to the trivial element. \repr{mingen} thus implies that $\operatorname{NG}(P_{n}(\rp))\geq \operatorname{G}((P_{n}(\rp))\up{Ab})=n$. In order to obtain \req{genspnrp2}, it suffices by equations~\reqref{inequgeni} and~\reqref{inequgen} to show that $\operatorname{TG}(P_{n}(\rp))\leq n$. We achieve this by proving that the set $Y$ described by \req{torgenrp2} consists of torsion elements and that it generates $P_{n}(\rp)$. The first assertion follows immediately from the fact that $a$ and $b$ are of finite order (recall from \req{defalphabeta} that $a^{n}$ and $b^{n-1}$ are both of order $4$, hence all of the elements of $Y$ are of order $4$). It remains to show that $P_{n}(\rp)=\ang{Y}$. To do so, first observe that by the relation
\begin{equation*}
\rho_{j}^{-1}\rho_{i}^{-1}\rho_{j} \rho_{i}=A_{i,j}\quad\text{for all $1\leq i<j\leq n$}
\end{equation*}
given in~\cite[Lemma~17]{GG2}, it follows using \req{prespnrp} that $P_{n}(\rp)$ is generated by $\brak{\rho_{1},\ldots, \rho_{n}}$ (note that the elements denoted by $B_{i,j}$ in~\cite{GG2} are the elements $A_{i,j}$ of this paper). From \req{defalphabeta}, $AB^{-1}=\rho_{n}$, and applying \req{conja}, we see that for $j=0,1,\ldots,n-1$,
\begin{equation}\label{eq:rhoconjb}
a^{-j} (AB^{-1}) a^j=  A\ldotp a^{-j}B^{-1} a^j= \rho_{n-j}.
\end{equation}
Thus $\ang{A,\, a^{-j} B a^j,\, j=0,1,\ldots, n-1}= \ang{\rho_{1},\ldots, \rho_{n}}=P_{n}(\rp)$. But $A=\rho_{n}\cdots \rho_{1}$ by \req{defalphabeta}, and so
\begin{equation*}
a^{-(n-1)} B^{-1}a^{n-1}=A^{-1}\rho_{1}=A^{-1} (\rho_{2}^{-1}\cdots \rho_{n}^{-1})A=A^{-1} \left(\prod_{j=0}^{n-2} \left( a^{n-2-j} B a^{j+2-n} \ldotp A^{-1}\right)\right)\ldotp A
%
%
%
%
%A\ldotp a^{-(n-1)} B^{-1}a^{n-1}=\rho_{1}=\left(\prod_{j=0}^{n-2} \left( a^{n-2-j} (AB^{-1})^{-1} a^{j+2-n} \right)\right)\ldotp A.
\end{equation*}
by \req{rhoconjb}. We can thus remove $a^{-(n-1)} B^{-1} a^{n-1}$ from the list of generators, and so $P_{n}(\rp)=\ang{Y}$, which proves the first part of the statement. Since the elements of $Y$ are of order $4$, we obtain the inequality $\operatorname{TG}(P_{n}(\rp))\leq n$ as required.\qedhere 
\end{enumerate}
\end{proof}

We finish this section by computing the number of conjugacy classes in $B_{n}(\rp)$ of elements of order $4$ lying in $P_{n}(\rp)$. Before doing so, we determine the centraliser of the elements $a$ and $b$ in $B_{n}(\rp)$. If $\Gamma$ is a group and $y\in \Gamma$, we denote the centraliser of $y$ in $\Gamma$ by $Z_{\Gamma}(y)$, and the centre of $\Gamma$ by $Z(\Gamma)$. A sketch of the following result appeared in~\cite[Remark~24]{GG10}. Note that line~5 of the final paragraph of that remark should read `are maximal finite cyclic subgroups', and not `are maximal finite subgroups'.

\begin{prop}\label{prop:centx}
Let $n\geq 2$, let $\widehat{x}\in \brak{a,b}$, and let $y\in \brak{x,\widehat{x}}$, where $x=a^{n}$ (resp.\ $y=b^{n-1}$) if $\widehat{x}=a$ (resp.\ $\widehat{x}=b$). Then $Z_{B_{n}(\rp)}(y)=\ang{\widehat{x}\mspace{1mu}}$.
\end{prop}

\begin{proof}
Let $\widehat{x},x$ and $y$ be as in the statement. By~\cite[Corollary~4]{GG10}, $Z_{P_{n}(\rp)}(x)=\ang{x}$ is finite (the result also holds if $n=2$ since $P_{2}(\rp)\cong \quat$). It follows from the short exact sequence~\reqref{defperm} that $Z_{B_{n}(\rp)}(x)$ is finite, and the inclusion $Z_{B_{n}(\rp)}(\widehat{x}) \subset Z_{B_{n}(\rp)}(x)$ then implies that $Z_{B_{n}(\rp)}(y)$ is finite. Furthermore, $Z(Z_{B_{n}(\rp)}(y))\supset \ang{y}$, so the order of $Z(Z_{B_{n}(\rp)}(y))$ is at least $4$. On the other hand, $Z_{B_{n}(\rp)}(y)$ clearly contains the cyclic group $\ang{\widehat{x}\mspace{1mu}}$ of order $4(n-k)$, where $k=0$ (resp.\ $k=1$) if $\widehat{x}=a$ (resp.\ $\widehat{x}=b$), which using~\cite[Theorem~5]{GG12} and the subgroup structure of dicyclic and the binary polyhedral groups (for the latter, see~\cite[Proposition~85]{GG13} for example) is maximal as a finite cyclic group. Suppose that $Z_{B_{n}(\rp)}(y)\neq \ang{\widehat{x}\mspace{1mu}}$. This subgroup structure implies that $Z_{B_{n}(\rp)}(y)$ would then be either dicyclic or binary polyhedral, and so its centre $Z(Z_{B_{n}(\rp)}(y))$ would be isomorphic to $\Z_{2}$. To see this, note that if $G$ is a dicyclic (resp.\ binary polyhedral) group then it possesses a unique element $g$ of order $2$ that is central. If $G\cong \quat$ then clearly $Z(G)\cong \Z_{2}$. So assume that $G\ncong \quat$. Then $G/\ang{g}$ is dihedral of order at least $6$ (resp.\ isomorphic to one of $\an[4]$, $\sn[4]$ and $\an[5]$), and so its centre is trivial. It thus follows that $Z(G)=\ang{g}$, and so $Z(G)\cong \Z_{2}$. But $y\in Z(Z_{B_{n}(\rp)}(y))$ by construction, and taking $G=Z_{B_{n}(\rp)}(y)$ yields a contradiction since $y$ is of order at least $4$. We thus conclude that $Z_{B_{n}(\rp)}(\widehat{x})= \ang{\widehat{x}\mspace{1mu}}$ as required.
\end{proof}

\begin{rem}
In the proof of \repr{centx}, we use~\cite[Corollary~4]{GG10}, whose proof depends (via~\cite[Proposition~21]{GG10} and~\cite[Proposition~17]{GG10} in the case $n=3$) upon the characterisation of the conjugacy classes of elements of order $4$ given in~\cite[Proposition~21(3)]{GG2}. During the discussion leading to \reth{genrp}(\ref{it:genrpd}), we realised that the statement of the latter result is not correct, and that there are in fact five conjugacy classes of elements of order $4$ in $B_{3}(\rp)$ (which give rise to four conjugacy classes of subgroups isomorphic to $\Z_{4}$), and not three. This does not however affect the validity of \cite[Corollary~4]{GG10}. An erratum for~\cite[Proposition~21(3)]{GG2} will appear elsewhere.
\end{rem}

We then have the following result, which answers the question posed in~\cite[Remark~22]{GG10}.

\begin{prop}\label{prop:ccpnrp2}
If $n\geq 2$ then the number of conjugacy classes in $B_{n}(\rp)$ of pure braids of order $4$ is equal to $(n-2)!(2n-1)$.
\end{prop}

\begin{proof}
Let $\brak{\gamma_{1},\ldots,\gamma_{n!}}$ be a set of coset representatives of $P_{n}(\rp)$ in $B_{n}(\rp)$, let $w\in P_{n}(\rp)$ be of order $4$, and let $[w]$ denote the conjugacy class of $w$ in $P_{n}(\rp)$. By~\cite[Proposition 10]{GG10} and the short exact sequence~\reqref{defperm}, there exist $1\leq i\leq n!$, $\epsilon\in\brak{1,-1}$, $z\in P_{n}(\rp)$ and $x\in \brak{a^{n}, b^{n-1}}$ such that $w=z\gamma_{i} x^{\epsilon} \gamma_{i}^{-1}z^{-1}$. Now~\cite[Proposition~15]{GG12} implies that $\widehat{x}$ is conjugate to $\widehat{x}^{-1}$ in $B_{n}(\rp)$, where $\widehat{x}=a$ (resp.\ $\widehat{x}=b$) if $x=a^{n}$ (resp.\ $x=b^{n-1}$). Hence $x$ and $x^{-1}$ are conjugate in $B_{n}(\rp)$, and by picking a different coset representative $\gamma_{i}$ if necessary, we may suppose that $\epsilon=1$, whence $[w]=[\gamma_{i} x \gamma_{i}^{-1}]$. Conversely each set $[\gamma_{i} x \gamma_{i}^{-1}]$ represents a conjugacy class of elements of order $4$ in $P_{n}(\rp)$. We must thus count the number of such conjugacy classes in $P_{n}(\rp)$. First note that since $a^{n}$ and $b^{n-1}$ are sent to distinct elements of $(B_{n}(\rp))\up{Ab}\cong \Z_{2}\oplus \Z_{2}$ under Abelianisation, $[\gamma_{i} a^{n} \gamma_{i}^{-1}]\neq [\gamma_{j} b^{n-1} \gamma_{j}^{-1}]$ for all $1\leq i,j\leq n$. It thus suffices to enumerate the number of distinct conjugacy classes $[\gamma_{i} x \gamma_{i}^{-1}]$ for each $x\in \brak{a^{n}, b^{n-1}}$. Set $k=0$ (resp.\ $k=1$) if $x=a^{n}$ (resp.\ $x=b^{n-1}$). If $1\leq i,j\leq n!$, we have:
\begin{align*}
[\gamma_{i} x \gamma_{i}^{-1}]\!=\![\gamma_{j} x \gamma_{j}^{-1}] & \Longleftrightarrow \text{$\exists h\in P_{n}(\rp)$ such that $h\gamma_{i} x \gamma_{i}^{-1}h^{-1}=\gamma_{j} x \gamma_{j}^{-1}$}\notag\\
& \Longleftrightarrow \text{$\exists h\in P_{n}(\rp)$ such that $\gamma_{j}^{-1}h\gamma_{i}$ commutes with $x$}\notag\\
& \Longleftrightarrow \text{$\exists h\in P_{n}(\rp), \, \exists l\!\in\! \brak{0,1,\ldots,4(n-k)-1}$ such that $\gamma_{j}^{-1}h\gamma_{i}\!=\!\widehat{x\,}^{l}$}\\
& \Longleftrightarrow \text{$\exists h'\in P_{n}(\rp), \, \exists l\!\in\! \brak{0,1,\ldots,4(n-k)-1}$ such that $h'\gamma_{j}^{-1}\gamma_{i}\!=\!\widehat{x\,}^{l}$}\\
& \Longleftrightarrow \text{$\exists l\in \brak{0,1,\ldots,4(n-k)-1}$ such that $\widehat{x\,}^{l}\gamma_{i}^{-1}\gamma_{j}\in P_{n}(\rp)$}\\
& \Longleftrightarrow \text{$\exists l\in \brak{0,1,\ldots,4(n-k)-1}$ such that $\pi(\gamma_{i})=\pi(\gamma_{j})(\pi(\widehat{x}))^{l}$}\\
& \Longleftrightarrow \text{$\pi(\gamma_{i})\equiv\pi(\gamma_{j})$ modulo $\ang{\pi(\widehat{x})}$,}
\end{align*}
where we have applied \repr{centx} to prove the equivalence of the second and third lines. This implies that the number of distinct conjugacy classes of the form $[\gamma_{i}x\gamma_{i}^{-1}]$ is equal to $[\sn:\ang{\pi(\widehat{x})}]$. But $\pi(\widehat{x})$ is an $(n-k)$-cycle, and so there are $n!/(n-k)$ such conjugacy classes. Taking the sum over $x\in\brak{a,b}$, it follows that the number of conjugacy classes in $B_{n}(\rp)$ of pure braids of order $4$ is given by:
\begin{equation*}
\frac{n!}{n}+\frac{n!}{n-1}=\frac{(2n-1)n!}{n(n-1)}=(n-2)!(2n-1)
\end{equation*}
as required.
\end{proof}

\section{The mapping class groups of $M=\dt, \St, \rp$}\label{sec:mcg}

Let $M$ be a compact, connected surface, perhaps with boundary, and let $X$ be an $n$-point subset lying in the interior of $M$. If $M$ is orientable (resp.\ non-orientable), we recall that the \emph{$n\up{th}$ mapping class group} of $M$, denoted by $\mcg$, is defined to be the set of isotopy classes of orientation-preserving homeomorphisms (resp.\ homeomorphisms) of $M$ that leave $X$ invariant and are the identity on the boundary, and where the isotopies leave both $X$ and the boundary fixed pointwise. Note that up to isomorphism, $\mcg$ does not depend on the choice of marked points $X$. The \emph{$n\up{th}$ pure mapping class group} of $M$, denoted by $\pmcg$, is the normal subgroup of $\mcg$ given by imposing the condition that the homeomorphisms fix the set $X$ pointwise. This gives rise to a short exact sequence similar to that of~\reqref{defperm}:
\begin{equation}\label{eq:defpermmcg}
1 \to \pmcg \to\mcg \stackrel{\widehat{\pi}}{\to} \sn \to 1.
\end{equation}
In the case where $M$ is the $2$-disc $\dt$, it is well known that $\mcgd\cong B_{n}$~\cite{Bi}. If we relax the hypothesis in the definition of $\mcgd$ that the isotopies fix the boundary then we obtain a group that we denote by $\mcgdb$, which is isomorphic to $B_{n}/\ang{\ft}$ (if $n\geq 3$ then $Z(B_{n})=\ang{\ft}$). We let $\pmcgdb$ denote the associated pure mapping class group.

Let $n\geq 2$. A first application of our results to $\mcgdb$ as well as to $\mcg$, $M$ being $\St$ or $\rp$, is obtained using the fact that the mapping class group in question is the quotient of the corresponding braid group by $\ang{\ft}$~\cite{Bi2,Ma,MKS,Sc}. Further, we have the following commutative diagram of short exact sequences:
\begin{equation*}
\xymatrix{%
& 1\ar[d] & 1\ar[d] & &\\
& \ang{\ft}\ar[d]\ar@{=}[r] &\ang{\ft} \ar[d] & &\\
1 \ar[r]  & P_{n}(M) \ar[r] \ar[d] & B_{n}(M) \ar[d] \ar[r]^{\pi} & \sn
\ar@{=}[d] \ar[r] & 1\\
1 \ar[r] & \pmcg \ar[d]\ar[r] & \mcg \ar[d]\ar[r]^(.65){\widehat{\pi}} & \sn \ar[r] & 1,\\
& 1 & 1 & &}
\end{equation*}
where the two vertical short exact sequences are obtained by taking the quotient of the (pure) braid groups of $M$ by $\ang{\ft}$. A similar diagram may be constructed for $B_{n}$ and $\mcgdb$.

Let us start by considering the case of the mapping class groups $\mcgdb$ and $\pmcgdb$ of the disc. 
\begin{prop}\label{prop:mingenartinmapp}
Let $n\geq 2$.
\begin{enumerate}[(a)]
\item We have:
\begin{equation*}
\operatorname{G}(\mcgdb)=\operatorname{TG}(\mcgdb)=
\begin{cases}
1 & \text{if $n=2$}\\
2 & \text{if $n\geq 3$.}
\end{cases}
\end{equation*}
Furthermore, $\operatorname{NG}(\mcgdb[2])=\operatorname{NTG}(\mcgdb[2])=1$, and
\begin{equation*}
\operatorname{NG}(\mcgdb)=1\quad\text{and}\quad\operatorname{NTG}(\mcgdb)=2\quad \text{for all $n\geq 3$}.
\end{equation*}

\item $\operatorname{NG}(\pmcgdb)= \operatorname{G}(\pmcgdb)= \operatorname{G}(\pmcgdb\up{Ab})=n(n-1)/2-1$.
\end{enumerate}
\end{prop}

\begin{rem}
Let $n\geq 2$. Then the group $\pmcgdb\cong P_{n}/\ang{\ft}$ is torsion free. To see this, first note that if $n=2$, we have $P_{2}=\ang{\ft[2]}=\ang{\sigma_{1}^{2}}$, and the result is clear. So suppose that $n\geq 3$. Recall from~\cite[Theorem~4(ii)]{GG1} that the projection $P_{n}\to P_{2}$ given geometrically by forgetting all but the first two strings gives rise to the isomorphism $P_{n}\cong \Z\oplus P_{n-2}(\dt\setminus \brak{x_1, x_2})$, where the $\Z$-factor is generated by $\ft$. Thus $P_{n}/\ang{ft}\cong P_{n-2}(\dt\setminus \brak{x_1, x_2})$, which is well known to be torsion free.
\end{rem}

\begin{proof}[Proof of \repr{mingenartinmapp}.]
If If $n=2$ then the statements follow easily because $\mcgdb[2]\cong B_{2}/\ang{\ft[2]}=\ang{\sigma_{1}}/\ang{\sigma_{1}^{2}}\cong \Z_{2}$ and $\pmcgdb[2]$ is trivial. So from now on, we suppose that $n\geq 3$.
\begin{enumerate}
\item The equality $\operatorname{G}(\mcgdb)=2$ follows from Propositions~\ref{prop:mingenartin}(\ref{it:gensbn}) and~\ref{prop:mingen}, as well as the surjectivity of the homomorphisms $B_{n}\to \mcgdb$ and $\map{\widehat{\pi}}{\mcgdb}[\sn]$, the second implying that $\mcgdb$ is non cyclic.

To see that $\operatorname{TG}(\mcgdb)=2$, the proof of \repr{mingenartin}(\ref{it:gensbn}) implies that $B_{n}=\ang{\sigma_{n-1}, \alpha_{0}}$, where $\alpha_0=\sigma_1\cdots\sigma_{n-1}$, but $\sigma_{n-1}=\alpha_{0}^{-1}\alpha_{1}$, where $\alpha_1=\sigma_1\cdots\sigma_{n-1}^{2}$, so $B_{n}=\ang{\alpha_{0},\alpha_{1}}$, and thus $\mcgdb=\ang{\overline{\alpha_{0}}, \overline{\alpha_{1}}}$, where $\overline{\alpha_{0}}$ and $\overline{\alpha_{1}}$ denote the $\ang{\ft}$-cosets of $\alpha_{0}$ and $\alpha_{1}$ respectively. But $\overline{\alpha_{0}}$ and $\overline{\alpha_{1}}$ are of finite order in $\mcgdb$, of order $n$ and $n-1$ respectively (see~\cite{E,K} or~\cite[Lemma~3.1]{GW}), and so individually do not generate $\mcgdb$. The result then follows by applying the inequality~\reqref{inequgeni}.

The equality $\operatorname{NG}(\mcgdb)=1$ is a straightforward consequence of the second part of \repr{mingenartin}(\ref{it:gensbn}), that $B_{n}=\normcl{\sigma_{1}}$, and \repr{mingen}. It remains to show that $\operatorname{NTG}(\mcgdb)=2$. To do so, let $\alpha\in \mcgdb$ be an element of finite order, and let $\map{\phi}{\mcgdb}[(\mcgdb)\up{Ab}]$ denote Abelianisation. Identifying $\mcgdb$ with $B_{n}\left/\ang{\ft}\right.$, and using the fact that
\begin{equation}\label{eq:ft}
\ft=(\sigma_{1}\cdots \sigma_{n-1})^{n},
\end{equation}
as well as the standard presentation~\reqref{Artin1} and~\reqref{Artin2} of $B_{n}$, we see that $(\mcgdb)\up{Ab}\cong \Z_{n(n-1)}$ is generated by the element $\phi(\sigma_{j})$ for all $j=1,\ldots,n-1$. Applying~\cite[Lemma~3.1]{GW}, it follows that there exists $i\in \brak{0,1}$ such that $\alpha\in \normcl{\overline{\alpha_{i}}}$, so $\normcl{\alpha}\subset \normcl{\overline{\alpha_{i}}}$, and hence $\ang{\phi(\alpha)}\subset \ang{\phi(\overline{\alpha_{i}})}$. But $\phi(\overline{\alpha_{i}})=\overline{n-1+i}$ in $\Z_{n(n-1)}$, so $\ang{\phi(\alpha)}\subset \ang{\phi(\overline{\alpha_{i}})}\subsetneqq (\mcgdb)\up{Ab}$. We conclude that $\normcl{\alpha}\subsetneqq \mcgdb$, and $\operatorname{NTG}(\mcgdb)>1$. The result then follows from \req{inequgen} and the previous paragraph.

\item From the proof of \repr{mingenartin}(\ref{it:gensbnb}), $H_{1}(P_n)=P_n\up{Ab}$ is a free Abelian group of rank $n(n-1)/2$ with a basis comprised of the classes of the elements $A_{i,j}$ given by \req{defaij}, where $1\leq i<j\leq n$. Recall that if $1\to K\to G\to Q\to 1$ is an extension of groups then we have a $6$-term exact sequence 
\begin{equation*}%\label{eq:stallings}
H_2(G)\to H_2(Q)\to K/[G,K]\to H_1(G)\to H_1(Q)\to 1
\end{equation*}
due to Stallings~\cite{Br,St}. Thus the central short exact sequence  $1 \to \ang{\ft}  \to  P_n \to  \pmcgdb \to 1$ gives rise to the following exact sequence: 
\begin{equation*}
H_{2}(P_{n})\to H_2(\pmcgdb) \to \ang{\ft} \stackrel{\psi}{\to} H_1(P_n)  \to H_1(\pmcgdb) \to 1,
\end{equation*}
where $\map{\psi}{\ang{\ft}}[H_1(P_n)]$ is the homomorphism induced by the inclusion $\ang{\ft}\to  P_n$. Now 
\begin{equation}\label{eq:ftdecomp}
 \ft=\prod_{i=1}^{n-1} (A_{i,i+1}\cdots A_{i,n}),
\end{equation}
and using the description of $P_n\up{Ab}$, we see that $\im{\psi}$ is generated by a primitive element of $P_n\up{Ab}$. The isomorphism $H_{1}(P_{n})/\im{\psi}\cong H_1(\pmcgdb)$ then implies that $\operatorname{G}(\pmcgdb\up{Ab})=n(n-1)/2-1$, and thus 
\begin{equation}\label{eq:pmcgdbineq}
\operatorname{G}(\pmcgdb)\geq \operatorname{NG}(\pmcgdb) \geq n(n-1)/2-1
\end{equation}
by \repr{mingen}. Using \req{ftdecomp} once more, the quotient $\pmcgdb\cong P_{n}/\ang{\ft}$ admits a generating set consisting of all but one of the $\ang{\ft}$-cosets of the $A_{i,j}$, and thus $(n-1)/2-1\geq \operatorname{G}(\pmcgdb)$. The result then follows from \req{pmcgdbineq}.\qedhere     
\end{enumerate}
\end{proof}

As we mentioned above, if $n\geq 2$, $\mcgst$ (resp.\ $\pmcgst$) is isomorphic to $B_n(\St)/\ang{\ft}$ (resp.\ $P_n(\St)/\ang{\ft}$). We obtain a result similar to that of parts~(\ref{it:bemaplusb}) and~(\ref{it:bemaplusd}) of \reth{bemaplus}.
%Since the group  $MCG( \St, n)$ is a   quotient of $MCG(D, n)$, the statement of Proposition 15 
%remains true if we replace nt of $MCG(D, n)$ by  $MCG( \St, n)$. 
%Now we look at the question for elements of torsion.
%  From \cite{GG3} it was proved 
%that the  set $S=\{\alpha_0, \alpha_1\}$  is a set of generators of $B_n(\St)$
%and both elements $\alpha_0$ and $ \alpha_1$ are torsion elements. So the same result is true for the mapping class, i.e. the mapping class group is %generated by the image of these two elements which are torsion.  In  \cite{BeMa} Proposition 5.3 it was shown that for $n$ odd the normal closure of 
%the element $\alpha_1$ is the group $B_n(\St)$. On the other-hand for $n$ even, we will see that 
%the normal closure of this  element is a normal  subgroup which has index 2.  
%Now we claim:

%\comment{Part~(a) has been added. For $B_{3}(\St)$ there was a special case of part~(c) for $i=2$, and this has been added here.}

\begin{prop}\mbox{}\label{prop:normclst}
\begin{enumerate}
\item\label{it:normclstodd} If $n\geq 3$ is odd then $\normcl{\overline{\alpha_{1}}}=\mcgst$. In particular, $\mcgst$ is strongly $(n-1)$-torsion generated.
\item\label{it:normclsta} If $n\geq 4$ is even, there is no torsion element in $\mcgst$ whose normal closure is equal to $\mcgst$. Further, the quotient $\mcgst\left/\normcl{\overline{\alpha_1}}\right.$ is isomorphic to $\Z_2$.
\item For all $n\geq 3$ and for $i=0,2$, the quotient $\mcgst\left/\normcl{\overline{\alpha_i}}\right.$ is isomorphic to $\Z_{n-1}$, unless $n=3$ and $i=2$, in which case $\mcgst[3]/\normcl{\overline{\alpha_{2}}} \cong \mcgst[3]\cong \sn[3]$.
\end{enumerate}
\end{prop}

\begin{proof}
For $i=0,1,2$, let $\alpha_{i}$ be the finite order element of $B_{n}(\St)$ defined in the statement of \reth{murasugi}. By~\cite{GvB}, the image $\overline{\alpha_{i}}$ of $\alpha_{i}$ under the projection $B_{n}(\St)\to \mcgst$ is of order $n-i$. Since $\alpha_{i}^{n-i}=\ft$, $\ft\in \normcl{\alpha_{i}}$, and we obtain the following commutative diagram of short exact sequences:
\begin{equation}\label{eq:commdiagst}
\begin{xy}*!C\xybox{%
\xymatrix{%
& 1\ar[d] & 1\ar[d] & &\\
& \ang{\ft}\ar[d]\ar@{=}[r] &\ang{\ft} \ar[d] & &\\
1 \ar[r]  & \normcl{\alpha_{i}} \ar[r] \ar[d] & B_{n}(\St) \ar[d] \ar[r] & B_{n}(\St)\left/\normcl{\alpha_{i}}\right. \ar[d]^{\cong} \ar[r] & 1\\
1 \ar[r] & \normcl{\overline{\alpha_{i}}} \ar[d]\ar[r] & \mcgst \ar[d]\ar[r] & \mcgst\left/\normcl{\overline{\alpha_{i}}}\right. \ar[r] & 1.\\
& 1 & 1 & &}}
\end{xy}
\end{equation}
The proposition then follows from \reth{bemaplus}(\ref{it:bemaplusb}) and~(\ref{it:bemaplusd}), using the fact in part~(\ref{it:normclsta}) that every torsion element of $\mcgst$ is conjugate to a power of one of the $\overline{\alpha_{i}}$(\cite{E,K} or~\cite[Lemma~3.1]{GW}).
\end{proof}

For the mapping class groups of $\St$, we thus obtain results similar to those of \reth{bemaplus}(\ref{it:bemaplusa})~(\ref{it:bemaplusc}) and~(\ref{it:bemapluse}). In the case of the pure mapping class groups, \req{dirsum} implies that $\pmcgst\cong P_{n}(\St)\left/\ang{\ft}\right.$ is isomorphic to $P_{n-3}(\St)$, which is torsion free.

\begin{thm}
Let $n\geq 3$.
\begin{enumerate}
\item\label{it:gensmcgst} $\operatorname{G}(\mcgst)=2$,  $\operatorname{NG}(\mcgst)=1$ and $\operatorname{TG}(\mcgst)=2$.
\item $\operatorname{NTG}(\mcgst)=
\begin{cases}
1 & \text{if $n$ is odd}\\
2 & \text{if $n$ is even.}
\end{cases}$
\item $\operatorname{G}(\pmcgst)= \operatorname{NG}(\pmcgst)=\operatorname{G}(\pmcgst)\up{Ab}=n(n-3)/2$.
\end{enumerate}
\end{thm}

\begin{proof}\mbox{}
\begin{enumerate}
\item The equalities follow from \reth{bemaplus}(\ref{it:bemaplusa}), using \req{inequgeni}, \repr{mingen} applied to the epimorphism $B_{n}(\St)\to \mcgst$, and the fact that $\mcgst$ is non cyclic.

\item If $n$ is odd then the result is given by \repr{normclst}(\ref{it:normclstodd}). So suppose that $n$ is even. \repr{normclst}(\ref{it:normclsta}) implies that $\operatorname{NTG}(\mcgst)>1$, and the result then follows from part~(\ref{it:gensmcgst}) and \req{inequgen}.

\item From \req{dirsum} and \reth{bemaplus}(\ref{it:bemapluse}), we have that $\pmcgst\cong P_{n-3}(\St)$ and $\operatorname{G}(\pmcgst)=n(n-3)/2$. We saw in the proof of \reth{bemaplus}(\ref{it:bemapluse}) that
\begin{equation*}
(P_{n-3}(\St \setminus \brak{x_1,x_2,x_3}))\up{Ab}\cong \Z^{n(n-3)/2}, 
\end{equation*}
so $\operatorname{G}(\pmcgst)\up{Ab}=n(n-3)/2$. The remaining equality is a consequence of \repr{mingen}.\qedhere
\end{enumerate}
\end{proof}

%For the group $\pmcgst$ the situation is simpler. This group has no torsion and it is isomorphic to $P_{n-3}(\St)$. So we can state:
%we can not expect that it can be generated by a set of torsion elements
%since this group contains only one non trivial torsion element which is the full twist. The number   of generating  
%and  closure  generating  of $PMCG(\St, n)$  is given as follows.
 
We now turn to the case of the mapping class groups of the projective plane.  Let $n\geq 2$. If $x\in B_{n}(\rp)$, let $\overline{x}$ denote its image in the quotient of $B_{n}(\rp)$ by $\ang{\ft}$, which we identify with $\mcgrp$. The Abelianisation of $B_{n}(\rp)$ is isomorphic to $\Z_{2}\oplus \Z_{2}$, where the first (resp.\ second) factor identifies all of the generators $\sigma_{i}$ (resp.\ $\rho_{j}$) of the presentation given by \repr{presvb}. Since the Abelianisation of $\ft$ is trivial by \req{ft}, the fact that $\mcgrp\cong B_{n}(\rp)\left/\ang{\ft}\right.$ implies that the Abelianisation of $\mcgrp$ is also isomorphic to $\Z_{2}\oplus \Z_{2}$, where the first (resp.\ second) factor identifies all of the generators $\overline{\sigma_{i}}$ (resp.\ $\overline{\rho_{j}}$) of $\mcgrp$.

\begin{prop}\label{genrpmapp}
Let $n\geq 2$.
\begin{enumerate}
\item\label{it:torsgenmcgrp} Let $a$ and $b$ be as defined in \req{defab}. Then the group $\mcgrp$ is generated by the elements  $\overline{a}$ and $\overline{b}$, which are of order $2n$ and $2(n-1)$ respectively. 
\item\label{it:gensmcgrp} $\operatorname{G}(\mcgrp)=\operatorname{TG}(\mcgrp)=2$.
\item The normal closure of any element of $\mcgrp$  is a proper 
subgroup of $\mcgrp$, and $\operatorname{NTG}(\mcgrp)=\operatorname{NG}(\mcgrp)=2$. In particular,  $\mcgrp$ is not strongly $k$-torsion generated for any $k\in \N$.
\item The quotient of $\mcgrp$ by either $\normcl{\vphantom{\smash{\overline{b}}}\overline{a}}$ or $\normcl{\smash{\overline{b}}}$ is isomorphic to $\Z_2$.  
\end{enumerate}
\end{prop}

\begin{proof}\mbox{}
\begin{enumerate}
\item This follows immediately from \reth{genrp}(\ref{it:genrpa}), the fact that $a$ and $b$ are of order $4n$ and $4(n-1)$ respectively, as well as the uniqueness of $\ft$ as an element of order~$2$ of $B_{n}(\rp)$~\cite[Proposition~23]{GG2}.

\item Part~(\ref{it:torsgenmcgrp}) and \req{inequgeni} imply that $\operatorname{G}(\mcgrp)\leq\operatorname{TG}(\mcgrp)\leq 2$. On the other hand, $\mcgrp$ is not cyclic since its Abelianisation is $\Z_{2}\oplus \Z_{2}$, and so $\operatorname{G}(\mcgrp)>1$, which yields the result.

\item The first part follows from \repr{mingen} and the fact that the Abelianisation of  $B_n(\rp)$ is $\Z_2\oplus\Z_2$, which implies that $\operatorname{NG}(\mcgrp)>1$. The given equalities then follow from part~(\ref{it:gensmcgrp}) and equations~\reqref{inequgeni} and~\reqref{inequgen}.

\item The result is a consequence of \reth{genrp}(\ref{it:torsnormgenrp2}) and the commutative diagram~\reqref{commdiagst} with $\St$ replaced by $\rp$, and $\alpha_{i}$ replaced by either $a$ or $b$.\qedhere
%\item Use the fact that the quotient is cyclic, so is Abelian, and thus is a quotient of $\Z_2 \oplus \Z_2$, and is non trivial (since the image of the normal closure of $x$ or $y$ is different from $\Z_2 \oplus \Z_2$). , the short exact sequence 
%$1 \to \Z_2 \to  B_n(\rp) \to  \mcgrp \to 1$ and the fact that the induced map from Stalling sequence 
%$\Z_2 \to  (B_n(\rp))\up{Ab}$ is the trivial map
\end{enumerate}
\end{proof}

Finally, we study the situation for the pure mapping class of $\rp$.

\begin{prop}\label{prop:ntgmapp}
Let $n\in \N$. Then $\pmcgrp$ is torsion generated by the set of torsion elements $\brak{\overline{a}^{n}, \overline{b}^{n-1}, \overline{a} \overline{b}^{n-1}\overline{a}^{-1}, \cdots, \overline{a}^{n-2}\overline{b}^{n-1}\overline{a}^{2-n}}$ of order $2$, and
\begin{align*}
\operatorname{G}(\pmcgrp)&= \operatorname{NG}(\pmcgrp)=\operatorname{TG}(\pmcgrp)\\
&=\operatorname{NTG}(\pmcgrp)=n.
\end{align*}
In particular, $\pmcgrp$ cannot be normally generated by any subset containing less than $n$ elements.
\end{prop}

\begin{proof}
The first part is a consequence of the corresponding statement for $P_{n}(\rp)$ given in \reth{genrp}(\ref{it:genrpd}), the surjectivity of the homomorphism $P_{n}(\rp)\to \pmcgrp$ and the uniqueness of $\ft$ as an element of order~$2$ of $P_{n}(\rp)$. This implies that
\begin{equation*}
\operatorname{TG}(\pmcgrp)\leq n.
\end{equation*}
For the second part, recall from that proof of \reth{genrp}(\ref{it:genrpd}) that $(P_{n}(\rp))\up{Ab}\cong \Z_{2}^n$. Since the Abelianisation of $\ft$ is trivial, as in the case of $\mcgrp$ we have that $(P_{n}(\rp))\up{Ab}\cong (\pmcgrp)\up{Ab}$, which implies that $\operatorname{NG}(\pmcgrp)\geq n$ using \repr{mingen}. The given equalities then follow immediately from equations~\reqref{inequgeni} and~\reqref{inequgen}.\end{proof}

\section{The action of  $B_n(S)$ on the universal covering of $F_n(M)$, $M=\St, \rp$}\label{sec:actionh3}

%\comment{This has been reorganised a little.}
In this section we give another application of our results to the study of the action of the braid groups of $\St$ and $\rp$ on the homology of the universal covering of the associated configuration spaces. If $\Sigma^n$ is a finite-dimensional $CW$-complex that has the homotopy type of the $n$-sphere $\St[n]$ and if $G$ is a group that acts on $\Sigma^n$, it is interesting to know whether the homomorphism induced by each element of $G$ on $H_n(\Sigma^n;\Z)\cong \Z$ is trivial (\emph{i.e.}\ is the identity, $\id$) or not (\emph{i.e.}\ is $-\id$). As the example of the action of $\Z$ on $\St[1]\times R$ given by $t(z,n)=(\overline{z},  t+n)$ shows, $\overline{z}$ denoting complex conjugation, in general the induced homomorphism is non trivial. In this section, we will show that if $M=\St$ or $\rp$, the group $B_n(M)$ acts trivially on $H_3(\widetilde{F_n(M)}; \Z)$.

If $X$ is a path-connected space that admits a universal covering $\widetilde{X}$, it is well known that the fundamental group $\pi_1(X)$ of $X$ acts freely on $\widetilde{X}$. If $M=\St$ (resp.\ $M=\rp$) and $n\geq 3$ (resp.\ $n\geq 2$), the ordered and unordered configuration spaces $F_{n}(M)$ and $D_{n}(M)$ of $M$ are finite-dimensional manifolds of dimension $2n$, and their universal coverings $\widetilde{F_n(M)}$ and $\widetilde{D_n(M)}$, which in fact coincide, have the homotopy type of $\St[3]$~\cite{BCP,FZ,GG13} (resp.\ \cite{GG2}). 
%\comment{I added the references to~\cite{BCP,FZ}.} 
Although the \emph{method} of our proof of \repr{acthom} for the braid groups does not apply to the case of $P_{n}(\St)$ since this group is not torsion generated, the result itself is certainly true because $P_n(\St)$ is a subgroup of $B_n(\St)$. On the other hand, our method also applies directly to $P_n(\rp)$ since it is torsion generated.

\begin{proof}[Proof of \repr{acthom}.]
It suffices to prove the result for $H=B_n(M)$. From the proof of~\cite[Proposition~10.2, VII.10]{Br}, a finite-order element of a group $G$ that acts freely on a finite-dimensional homotopy (or homology) sphere $X$ whose homotopy type is that of  $\St[2n-1]$ acts trivially on the infinite cyclic group $H_{2n-1}(X;\Z)$. As we mentioned above, $\widetilde{F_n(M)}$ is a finite-dimensional complex that coincides with $\widetilde{D_n(M)}$ and that has the homotopy type of $\St[3]$. Further, $B_{n}(M)=\pi_{1}(D_n(M))$ acts freely on $\widetilde{D_n(M)}$, and $B_n(\St)$ (resp.\ $B_n(\rp)$) is torsion generated by~\cite[Theorem~3]{GG3} (resp.\ \reth{genrp}(\ref{it:genrpa})). Thus any element $x$ of $B_n(M)$ is a product of a finite number of elements of finite order, and so its action on $\widetilde{D_n(M)}$ induces the identity on $H_3(\widetilde{D_n(M)};\Z)$ because this action is the composition of homomorphisms, each of which is the identity by the result of~\cite{Br} mentioned above.
\end{proof}

\end{document}